\documentclass[12pt,reqno]{amsart}

\newcommand\version{September 8, 2019}

\usepackage{amsmath,amsfonts,amsthm,amssymb,amsxtra}
\usepackage{bbm} 

\newtheorem{theorem}{Theorem}
\newtheorem{proposition}[theorem]{Proposition}
\newtheorem{lemma}[theorem]{Lemma}

\theoremstyle{definition}

\newtheorem{remark}[theorem]{Remark}
\newtheorem{remarks}[theorem]{Remarks}

\theoremstyle{remark}

\newcommand{\1}{\mathbbm{1}}

\renewcommand{\epsilon}{\varepsilon}

\renewcommand{\phi}{\varphi}
\newcommand{\R}{\mathbb{R}}

\newcommand{\Sph}{\mathbb{S}}

\newcommand{\Z}{\mathbb{Z}}

\DeclareMathOperator{\diam}{diam}

\DeclareMathOperator{\supp}{supp}

\begin{document}

\title[Proof of spherical flocking --- \version]{Proof of spherical flocking \\ based on quantitative rearrangement inequalities}

\author{Rupert L. Frank}
\address[R. L. Frank]{Mathematisches Institut, Ludwig-Maximilans Univers\"at M\"unchen, The\-resienstr. 39, 80333 M\"unchen, Germany, and Munich Center for Quantum Science and Technology (MCQST), Schellingstr. 4, 80799 M\"unchen, Germany, and Mathematics 253-37, Caltech, Pasa\-de\-na, CA 91125, USA}
\email{rlfrank@caltech.edu}

\author{Elliott H. Lieb}
\address[E. H. Lieb]{Departments of Mathematics and Physics, Jadwin Hall, Princeton University, Princeton, NJ 08544, USA}
\email{lieb@princeton.edu}

\begin{abstract}
Our recent work on the Burchard--Choksi--Topaloglu flocking problem showed that in the large mass regime the ground state density profile is the characteristic function of some set. Here we show that this set is, in fact, a round ball. The essential mathematical structure needed in our proof is a strict rearrangement inequality with a quantitative error estimate, which we deduce from recent deep results of M. Christ.
\end{abstract}


\maketitle

\renewcommand{\thefootnote}{${}$} \footnotetext{\copyright\, 2019 by
  the authors. This paper may be reproduced, in its entirety, for
  non-commercial purposes.\\
  Partial support through US National Science Foundation grant DMS-1363432 and through German Research Foundation grant EXC-2111 390814868 (R.L.F.) is acknowledged.}


\section{Introduction and main result} 

We continue our study of the flocking (or swarming) model introduced by Burchard, Choksi and Topaloglu \cite{BCT} to describe the stable states of a large group of animals such as fish or birds. The model is described as follows. There is a function $\rho$ on $\R^3$ or, more generally, on $\R^N$, $N\geq 1$, which quantifies the density of animals and is allowed to take values between zero and one. (The upper bound one is there to prevent the animals from crashing into each other.) The total mass of animals is $m=\int_{\R^3}\rho(x)\,dx$ and is specified as a parameter in the problem.

The density $\rho$ must minimize the following `energy'
\begin{equation}
\label{eq:energy}
\mathcal E_{\alpha,\lambda}[\rho] = \frac12 \iint_{\R^N\times\R^N} \rho(x) \left( |x-y|^\alpha + \frac{1}{|x-y|^\lambda} \right) \rho(y)\,dx\,dy \,.
\end{equation}
Here $\alpha>0$ and $0<\lambda<N$ are parameters and the first term above represents an attractive `force' pulling the animals together and the second term is a repulsive `force' which keeps the animals apart. By using the different length scalings of the two terms in \eqref{eq:energy} we can ignore any need for a coupling constant in front of these terms. The power laws $|x|^\alpha$ and $|x|^{-\lambda}$ in \eqref{eq:energy} are not sacrosanct.

The minimization problem, formally stated, is
$$
E_{\alpha,\lambda}(m) = \inf\left\{ \mathcal E_{\alpha,\lambda}[\rho] :\ 0\leq \rho\leq 1 \,,\ \int_{\R^N} \rho(x)\,dx = m \right\}.
$$
The existence of a minimizer of the problem is shown in \cite{CFT}, but the qualitative features of this minimizer remained an interesting topic of investigation. Here we focus on the regime of large $m$.

Our main result is the following.

\begin{theorem}\label{main}
Let $0<\lambda<N-1$ and $\alpha>0$. Then there is an $m_{N,\alpha,\lambda}<\infty$ such that for all $m>m_{N,\alpha,\lambda}$ the only minimizers for $E_{\alpha,\lambda}(m)$ are characteristic functions of balls (up to sets of measure zero).
\end{theorem}

\begin{remarks}
(1) Our proof gives (in principle) a computable value of the constant $m_{N,\alpha,\lambda}$. In particular, we do not use compactness in our proof. However, the value is probably very far from being optimal.\\
(2) We do \emph{not} claim that when characteristic functions are minimizers, they are necessarily balls. It is conceivable that there is an intermediate range of $m$ for which minimizers are characteristic functions of sets which are not balls. This possibility, however, does not occur in the special case $\alpha=2$ and $\lambda=N-2$ in $N\geq 3$ \cite{BCT}.\\
(3) The assumption $\lambda<N-1$ is necessary, since one can show that if $N-1\leq\lambda<N$, then, although minimizers exist, balls can never be critical points of the minimization problem; see Remark \ref{failure}. Our experience with the ball suggests that for $N-1\leq\lambda<N$ minimizers are never characteristic functions.\\
(4) The theorem, in the special case $\alpha=2$, is due to Burchard, Choksi and Topaloglu \cite{BCT}. (There it is said to be valid for $N-1\leq\lambda<N$ as well, but this is incorrect as shown in Remark \ref{failure}.) The case $\alpha=2$, however, is rather special, since, writing $|x-y|^2= |x|^2-2x\cdot y+|y|^2$ and imposing, without loss of generality, the constraint of vanishing center of mass, we see that $\rho\mapsto\mathcal E_{2,\lambda}[\rho]$ is strictly convex. Thus, any solution of the Euler--Lagrange equation is necessarily the unique (up to translations) minimizer.\\
(5) The theorem can also be considered as known in the case $2<\alpha\leq 4$. Indeed, this follows from a remarkable convexity result of Lopes \cite{Lo} and the same method of proof as sketched in the previous remark.\\
(6) We have stated and proved the theorem for an interaction kernel of the form $|x|^\alpha +|x|^{-\lambda}$. However, as will be clear from our proof, the result holds for much more general kernels $k(|x|)$. The crucial assumptions are that $k(r)$ is increasing for large $r$ and tends to $+\infty$ as $r\to+\infty$, as well as that $k*\1_{E^*}$, where $E^*$ is the ball of volume $m$, is differentiable on the surface of $E^*$. (This is where the assumption $\lambda<N-1$ comes from.) However, we stick to the above model case to explain our ideas as clearly as possible.
\end{remarks}

Besides the cases mentioned in the above remarks, our Theorem \ref{main} is new. Also, since one cannot hope to use convexity outside the range $2\leq\alpha\leq 4$, a different method of proof than in these special cases is needed.

The tools we are using are quantitative rearrangement inequalities. To motivate those, let us note that the attractive and repulsive terms in the energy functional compete with each other. A simple scaling argument shows that in the large mass regime the attractive term is the dominant one. Moreover, notice that among all functions $0\leq\rho\leq 1$ with a given integral, the quantity
$$
\frac12 \iint_{\R^N\times\R^N} \rho(x)|x-y|^\alpha \rho(y)\,dx\,dy
$$
is minimal if and only if $\rho$ is the characteristic function of a ball. This follows from the Riesz rearrangement inequality (see \cite{Ri} and also \cite[Thm.~3.7]{LiLo}) and the bathtub principle (see \cite[Thm.~1.14]{LiLo}). One key ingredient in our proof is a lower bound on the energy gain when passing from $\rho$ to a characteristic function of a ball. As we will explain momentarily it is crucial that this lower bound is \emph{quadratic} in a certain distance of $\rho$ from balls. This quadratic rearrangement inequality is stated below as Theorem \ref{christcor2}. It is a rather straightforward consequence of a recent deep result by Christ \cite{Ch} concerning quadratic remainders in the Riesz rearrangement inequality.

Given the quadratic gain in the attractive term, our main work here concerns showing that there is at most a quadratic loss in the repulsive term. The large mass assumption corresponds by scaling to a small coupling constant in front of the repulsive term. Therefore, as soon as we have shown that the gain in the attractive term is at least quadratic and the loss in the repulsive term is at most quadratic, we will conclude that the minimizer $\rho$ is, in fact, the characteristic function of a ball when the mass is sufficiently large.

When trying to prove that the loss in the repulsive term is at most quadratic, the sense in which we measure the distance from a ball becomes crucial. For the quadratic gain in the attractive term, the distance is measured in $L^1$-norm. The difficulty that we face is that for the repulsive term, closeness in $L^1$-norm does \emph{not} guarantee a quadratic loss. The loss could be linear, for instance. One condition that does guarantee a quadratic bound on the loss is closeness in Hausdorff distance; see Proposition \ref{christcor1rev}.

Thus, our main work consists in showing that minimizers corresponding to $E_{\alpha,\lambda}(m)$ are close to balls in Hausdorff distance when $m$ is large. More precisely, we need to show that the set where the minimizer differs from being the characteristic function of a ball is confined to a shell around the surface of this ball and the relative width of this shell is comparable the relative $L^1$-distance from balls. This is achieved in Step 2 of the proof of Theorem \ref{main}, which is the technical heart of this paper. For the proof of this statement we use some ideas of Christ's proof of the quantitative Riesz rearrangement inequality. But due to the unboundedness of the repulsive potential $|x|^{-\lambda}$, we need to iterate his procedure on a large number of dyadic scales while carefully tracking the gains and losses on each scale.

\medskip

There is a certain similarity between the problem we treat here and a class of problems of the form
\begin{equation}
\label{eq:liquiddrop}
\mathrm{Per}_s(E) + \frac12 \iint_{E\times E} \frac{dx\,dy}{|x-y|^\lambda}
\end{equation}
which have recently received a lot of attention in the literature. Here $\mathrm{Per}_s(E)$ denotes the perimeter in the sense of De Giorgi for $s=1$ and its fractional generalization for $0<s<1$. The problem consists in minimizing the above functional over sets of a given measure $|E|=m$. In particular, for $s=1$, $\lambda=1$ and $N=3$ this is Gamow's famous liquid drop model for an atomic nucleus; see, e.g., \cite{FrLi,ChMuTo} and references therein. 

Due to the different scaling, the term $\mathrm{Per}_s(E)$ in \eqref{eq:liquiddrop} is now dominant for \emph{small} $m>0$. By the isoperimetric inequality and its fractional counterpart (see, e.g., \cite{FrSe}, based on ideas in \cite{AlLi}) among all sets $E$ of given measure, the term $\mathrm{Per}_s(E)$ is minimized precisely when $E$ is a ball (up to sets of measure zero). Using quantitative isoperimetric inequalities, it was shown in \cite{KnMu1,KnMu2,BoCh,Ju,FiFuMaMiMo} that for all sufficiently small masses, the only minimizers of \eqref{eq:liquiddrop} are balls (up to sets of measure zero). Thus, our result is an analogue of this result for the minimization problem $E_{\alpha,\lambda}(m)$.

Given the similarities in the problem and in the conclusion, there are, of course, also similarities between our method of proof and those applied in the context of \eqref{eq:liquiddrop}, but there are also important differences. As for similarities, for instance, the quantitative rearrangement inequality in Theorem \ref{christcor2} plays the role of the quantitative isoperimetric inequality. (As an aside we mention that Theorem \ref{christ} implies the quantitative isoperimetric inequality in the fractional case, see Proposition \ref{isoper}.) The proofs in \cite{KnMu1,KnMu2,BoCh,Ju,FiFuMaMiMo}, which show that small mass minimizers of \eqref{eq:liquiddrop} are balls, all rely to some degree on the well-developed regularity theory for almost minimizers of the (fractional) perimeter functional. Such theory is not available in our context and our proof of closeness in Hausdorff sense can be viewed as an initial step in this direction. We hope that it will turn out to be useful in other problems with non-local functionals as well.

\medskip

The remainder of this paper consists of two sections. The next one is devoted to quadratic estimates for single power kernels. We state our version of Christ's theorem (Theorem \ref{christ}) and deduce our lower bound on the gain in the attractive term (Theorem~\ref{christcor2}, as well as our upper bound on the loss in the repulsive term (Proposition~\ref{christcor1rev}). The following section contains the proof of Theorem \ref{main}, which we divide into a preparatory step and three main steps. The heart of the matter is Step 2, which proves closeness in Hausdorff sense.

\medskip

In the following we will frequently abbreviate
\begin{equation}
\label{eq:notationimu}
\mathcal I_\mu[\rho,\sigma] = \frac12 \iint_{\R^N\times\R^N} \rho(x) |x-y|^\mu \sigma(y)\,dx\,dy
\end{equation}
and $\mathcal I_\mu[\rho]=\mathcal I_\mu[\rho,\rho]$.

\subsection*{Acknowledgements}

The authors are grateful to A. Burchard and M. Christ for helpful remarks.


\section{Quadratic estimates}

\subsection{Christ's theorem and its consequences}

For a function $0\not\equiv\rho\in L^1(\R^N)$ with $0\leq\rho\leq 1$, we set
$$
A[\rho] := \left( 2\|\rho\|_1 \right)^{-1} \inf_{a\in\R^N} \| \rho-\1_{E^*+a} \|_1 \,,
$$
where $E^*$ is the ball, centered at the origin, of measure $|E^*|=\int_{\R^N} \rho\,dx$. One of the key tools in this paper is the following theorem.

\begin{theorem}\label{christ}
Let $0<\delta\leq 1/2$. Then there is a constant $c_{N,\delta}>0$ such that for all balls $B\subset\R^N$, centered at the origin, and all $\rho\in L^1(\R^N)$ with $0\leq\rho\leq 1$ and
$$
\delta \leq \frac{|B|^{1/N}}{2\, \|\rho\|_1^{1/N}} \leq 1-\delta \,,
$$
one has
$$
\frac12 \iint_{\R^N\times\R^N} \rho(x)\1_B(x-y)\rho(y)\,dx\,dy \leq \frac12 \iint_{E^*\times E^*} \1_B(x-y)\,dx\,dy - c_{N,\delta} \|\rho\|_1^2 A[\rho]^2 \,,
$$
where $E^*$ is the ball, centered at the origin, of measure $|E^*|=\int_{\R^N} \rho\,dx$.
\end{theorem}

In the case where $\rho$ is the characteristic function of a set $E$, this theorem is a special case of a more general result of Christ \cite{Ch} which concerns three possibly different sets. Let us explain this in more detail. In the case of three arbitrary sets the remainder involves two translation parameters and a matrix with determinant one. However, using the triangle inequality and the fact that in our case two of the sets coincide and the third one is a ball, one can bound this remainder term in terms of $A[\1_E]$, which involves only a single translation parameter. In this way, Theorem \ref{christ} for characteristic functions follows from the result in \cite{Ch}.

Our theorem for $\rho$ with values between zero and one is a modest extension of this theorem for characteristic functions which is obtained by essentially the same method of proof. A related extension appears in \cite{ChIl}, again in a three function setting. However, the inequality in \cite{ChIl} is of a somewhat different nature since objects defined on a compact Abelian group are compared with their rearrangements in $\R/\Z$.

We provide details of the proof of Theorem \ref{christ} in a supplementary note \cite{FrLi2} and do not claim any conceptual novelty compared with \cite{Ch}. We prepared this supplementary note for three reasons. First, the proof of Theorem \ref{christ}, as stated, does not appear in the literature. Second, we want to prove our claim that the constant $c_{N,\delta}$ is (in principle) computable and that no compactness is used. And third, we provide a somewhat different and more explicit treatment of the quadratic form which, in some sense, corresponds to the Hessian of the functional under consideration.

The overall strategy of Christ's proof in \cite{Ch} and our version of it bears some resemblance with the proof of a quantitative stability theorem for the Sobolev inequality by Bianchi and Egnell \cite{BiEg}, answering a question in \cite{BrLi}; see also \cite{ChFrWe}. This strategy consists of a first step which reduces the assertion to elements close to the set of optimizers and a second step where the inequality is proved close to the set of optimizers by a detailed analysis of the eigenvalues of the Hessian of the corresponding variational problem. Christ's analysis in \cite{Ch} is significantly more involved than this standard strategy and adds an additional step, since the metric in which closeness to the optimizers is measured in the first step ($L^1$ distance) and the metric in which the second step can be carried out (Hausdorff distance) are not equivalent.

We also note that quantitative stability theorems for Brascamp--Lieb--Luttinger inequalities \cite{BrLiLu} (see also \cite{Ro} together with \cite[Footnote 1]{ChON}), which generalize Riesz's rearrangment inequality to more than three functions, appear in \cite{Ch1,ChON}.

\medskip

As a consequence of Theorem \ref{christ} one obtains the following stability theorems for power-like kernels. We recall the notation \eqref{eq:notationimu}.

\begin{theorem}\label{christcor1}
Let $0<\lambda<N$. Then there is a $c_{N,\lambda}>0$ such that for all $\rho\in L^1(\R^N)$ with $0\leq\rho\leq 1$,
$$
\mathcal I_{-\lambda}[\rho] \leq \mathcal I_{-\lambda}[\1_{E^*}] - c_{N,\lambda} \|\rho\|_1^{2-\lambda/N} A[\rho]^2 \,,
$$
where $E^*$ is the ball, centered at the origin, of measure $|E^*|=\int_{\R^N} \rho\,dx$.
\end{theorem}

We do not need this theorem in our paper, but it might be useful elsewhere and is a simple consequence of Theorem \ref{christ}. We note that this result had been proved earlier in \cite{BC} for $N=3$, $\lambda=1$ and for characteristic functions $\rho$. For other $\lambda$ and $N$, the bounds from \cite{BC} involve $A[\rho]$ only with a non-optimal power larger than two.

\begin{proof}
We write
\begin{align*}
& \iint_{E^*\times E^*} \frac{dx\,dy}{|x-y|^\lambda} - \iint_{\R^N\times\R^N} \frac{\rho(x)\,\rho(y)}{|x-y|^\lambda}\,dx\,dy \\
& \quad = \lambda \int_0^\infty \frac{dR}{R^{\lambda+1}} \left( \iint_{E^*\times E^*} \1_{B_R}(x-y)\,dx\,dy -  \iint_{\R^N\times\R^N} \rho(x)\1_{B_R}(x-y)\rho(y)\,dx\,dy \right),
\end{align*}
where $B_R$ denotes the ball, centered at the origin, of radius $R$. Let
$$
I := \left\{ R>0 :\ \frac14 \leq \frac{|B_R|^{1/N}}{2\,\|\rho\|_1^{1/N}}\leq \frac34 \right\} \,.
$$
For $R\not\in I$, we bound the integrand of the $R$-integral from below by zero according to Riesz's theorem, while for $R\in I$ we apply Theorem \ref{christ} with $\delta=1/4$. We obtain
\begin{align*}
& \frac12 \iint_{E^*\times E^*} \frac{dx\,dy}{|x-y|^\lambda} - \frac12 \iint_{\R^N\times\R^N} \frac{\rho(x)\,\rho(y)}{|x-y|^\lambda}\,dx\,dy \geq \lambda\, c_{N,1/4}\, \|\rho\|_1^2\, A[\rho] \int_I \frac{dR}{R^{\lambda+1}} \,.
\end{align*}
A simple computation shows that the integral on the right side is a constant, depending on $N$ and $\lambda$, times $\|\rho\|_1^{-\lambda/N}$. This proves the theorem.
\end{proof}

The following theorem is the analogue for kernels involving positive powers. It plays a key role in our proof of Theorem \ref{main}.

\begin{theorem}\label{christcor2}
Let $\alpha>0$. Then there is a $c_{N,\alpha}>0$ such that for all $\rho\in L^1(\R^N)$ with $0\leq\rho\leq 1$,
$$
\mathcal I_\alpha[\rho] \geq \mathcal I_\alpha[\1_{E^*}] + c_{N,\alpha} \|\rho\|_1^{2+\alpha/N} A[\rho]^2 \,,
$$
where $E^*$ is the ball, centered at the origin, of measure $|E^*|=\int_{\R^N} \rho\,dx$.
\end{theorem}

\begin{proof}
The proof is similar to that of Theorem \ref{christcor1}, based on the formula
\begin{align*}
& \iint_{\R^N\times\R^N} \rho(x)|x-y|^\alpha\rho(y)\,dx\,dy - \iint_{E^*\times E^*} |x-y|^\alpha\,dx\,dy \\
& = \alpha \int_0^\infty \frac{dR}{R^{-\alpha+1}} \left( \iint_{E^*\times E^*} \1_{B_R}(x-y)\,dx\,dy -  \iint_{\R^N\times\R^N} \rho(x)\1_{B_R}(x-y)\rho(y)\,dx\,dy \right).
\end{align*}
We omit the details.
\end{proof}

As an aside before continuing with the main theme of our paper, let us show that Theorem \ref{christ}, together with the method from \cite{AlLi} as employed in \cite{FrSe}, yields the fractional isoperimetric inequality in quantitative form.

\begin{proposition}\label{isoper}
Let $N\geq 1$ and $0<s<1$. Then there is a $c_{N,s}>0$ such that for all measurable $E\subset\R^N$ of finite measure,
\begin{align*}
& \iint_{\R^N\times\R^N} \frac{|\1_E(x)-\1_E(y)|}{|x-y|^{N+s}}\,dx\,dy - \iint_{\R^N\times\R^N} \frac{|\1_{E^*}(x)-\1_{E^*}(y)|}{|x-y|^{N+s}}\,dx\,dy \\
& \qquad \geq c_{N,s} |E|^{(N-s)/N} A[\1_E]^2 \,.
\end{align*}
\end{proposition}

This inequality is weaker than \cite[Thm.~1.1]{FiFuMaMiMo} since our constant $c_{N,s}$ remains bounded as $s\to 1$, whereas that in \cite[Thm.~1.1]{FiFuMaMiMo} behaves like $(1-s)^{-1}$. Such a factor allows one to recover the classical isoperimetric inequality in the limit $s\to 1$.

\begin{proof}
As in \cite{AlLi,FrSe} we write
\begin{align*}
& \iint_{\R^N\times\R^N} \frac{|\1_E(x)-\1_E(y)|}{|x-y|^{N+s}}\,dx\,dy - \iint_{\R^N\times\R^N} \frac{|\1_{E^*}(x)-\1_{E^*}(y)|}{|x-y|^{N+s}}\,dx\,dy \\
& \quad = \int_0^\infty dt\, \left( \iint_{E^*\times E^*} \1_{\{ |x-y|^{-N-s}>t \} } \,dx\,dy - \iint_{E\times E} \1_{\{ |x-y|^{-N-s}>t \}} \,dx\,dy \right) \\
& \quad = (N+s) \int_0^\infty \frac{dR}{R^{N+s+1}}\, \left( \iint_{E^*\times E^*} \1_{B_R}(x-y) \,dx\,dy - \iint_{E\times E} \1_{B_R}(x-y) \,dx\,dy \right).
\end{align*}
(This identity is first verified by replacing $|x|^{-N-s}$ by an integrable kernel and then by passing to the limit.) Let $I$ be defined as in the proof of Theorem \ref{christcor1} with $\rho=\1_E$. For $R\not\in I$, we bound the integrand of the $R$-integral from below by zero according to Riesz's theorem, while for $\in I$ we apply Theorem \ref{christ} with $\delta=1/4$. This provides us with the lower bound
$$
2\,c_{N,1/4} (N+s)\, |E|^2\, A[\1_E]^2 \int_{I} \frac{dR}{R^{N+s+1}} = c_{N,s}\, |E|^{(N-s)/N}\, A[\1_E]^2 \,,
$$
as claimed.
\end{proof}


\subsection{Reverse inequalities}

We now prove an inequality which complements that in Theorem \ref{christcor1}. We emphasize that such inequalities can only hold if the set where the function deviates from a ball is confined to a small neighborhood of the surface of the ball.

\begin{proposition}\label{christcor1rev}
Let $0<\lambda<N-1$. Then there is a $C_{N,\lambda}<\infty$ such that for all $\rho\in L^1(\R^N)$ with
$$
\1_{(1-\theta)E^*} \leq \rho \leq \1_{(1+\theta)E^*} 
$$
for some $0\leq\theta\leq 1$, where $E^*$ is the ball of measure $\int_{\R^N} \rho\,dx$, one has
$$
\mathcal I_{-\lambda}[\rho] \geq \mathcal I_{-\lambda}[\1_{E^*}] - C_{N,\lambda} \|\rho\|_1^{2-\lambda/N} \theta^2 \,.
$$
\end{proposition}

We call this bound a quadratic estimate, since $\theta$ appears quadratically on the right side. We would like to apply this bound with $\theta = K\,A[\rho]$ for some constant $K\geq A[\rho]^{-1}$ and obtain
$$
\mathcal I_{-\lambda}[\rho] \geq \mathcal I_{-\lambda}[\1_{E^*}] - C_{N,\lambda} K^2 \|\rho\|_1^{2-\lambda/N} A[\rho]^2 \,,
$$
which complements the bound from Theorem \ref{christcor1}.

Related bounds appear in \cite[Lem.~5.3]{FiFuMaMiMo}, but there it is assumed that $\rho$ is the characteristic function of a star-shaped set with $C^1$ boundary. It is not clear whether this regularity assumption is satisfied in our case.

\begin{proof}
By scaling we may assume that $\int_{\R^N} \rho\,dx$ is the measure of the unit ball, and we will write $\mathcal B$ instead of $E^*$. We have, since convolution with $|x|^{-\lambda}$ is positive definite,
\begin{align*}
\mathcal I_{-\lambda}[\rho] - \mathcal I_{-\lambda}[\1_{\mathcal B}] & = \mathcal I_{-\lambda}[\rho-\1_{\mathcal B},\rho+\1_{\mathcal B}] \\
& = \mathcal I_{-\lambda}[\rho-\1_{\mathcal B},\rho+\1_{\mathcal B}] \\
& = 2 \, \mathcal I_{-\lambda}[\rho-\1_{\mathcal B},\1_{\mathcal B}] + \mathcal I_{-\lambda}[\rho-\1_{\mathcal B}] \\
& \geq 2 \, \mathcal I_{-\lambda}[\rho-\1_{\mathcal B},\1_{\mathcal B}] \,.
\end{align*}
Let
\begin{equation}
\label{eq:phi}
\phi(|x|) := \int_{\mathcal B} \frac{dy}{|x-y|^\lambda} \,.
\end{equation}
This notation is justified since the right side is a radial function. We now use the facts that $\rho-\1_{B}$ has integral zero, that $0\leq\rho\leq 1$ and that $\phi$ is non-increasing to write
$$
\mathcal I_{-\lambda}[\rho-\1_{\mathcal B},\1_{\mathcal B}] = \int_{\R^N} (\rho - \1_{\mathcal B}) (\phi-\phi(1))\,dx 
= - \int_{\R^N} \left|\rho-\1_{\mathcal B}\right| \left|\phi- \phi(1)\right|dx \,. 
$$
In Lemma \ref{phi} below we will show that
$$
|\phi(r) - \phi(1)| \leq C_{N,\lambda} |r-1|
\qquad\text{for all}\ r\geq 0 \,,
$$
with a constant $C_{N,\lambda}$ depending only on $N$ and $\lambda$. (We note that at this point the assumption $\lambda<N-1$ enters.) By assumption on $\rho$, we have
$$
|\rho-\1_{\mathcal B}| \leq \1_{\{ 1-\theta\leq |x|\leq 1+\theta \}}
$$
and therefore
$$
\mathcal I_{-\lambda}[\rho-\1_{\mathcal B},\1_{\mathcal B}] \geq - C_{N,\lambda} |\Sph^{N-1}| \int_{1-\theta}^{1+\theta} |r-1| r^{N-1}\,dr \geq - C_{N,\lambda}' \theta^2 \,,
$$
as claimed.
\end{proof}


\subsection{Bounds on potentials}

In this subsection we discuss the `$\lambda$-potential' $\phi$ of the unit ball $\mathcal B\subset\R^N$ defined in \eqref{eq:phi}. Similar bounds are stated, for instance, in \cite[Lem.~4.4]{KnMu2}, but since these bounds are important for us, we include some details.

\begin{lemma}\label{phi}
Let $0<\lambda<N-1$. Then $\phi$ is radial, strictly decreasing and its derivative bounded. Moreover, for some constant $C_{N,\lambda}$ depending only on $N$ and $\lambda$,
$$
|\phi(r) - \phi(1)| \leq C_{N,\lambda} |r-1|
\qquad\text{for all}\ r\geq 0 \,.
$$
\end{lemma}

\begin{proof}
The radial symmetry and the monotonicity are clear. The differentiability could be proved using bounds on Riesz potentials as in \cite[Thm.~10.2]{LiLo}, which treats the case $\lambda=N-2$. Here we use a different method which will also be useful in the following remark. We take the differentiability in the interior of $B$ and the exterior of $\overline B$ for granted and only show that these derivatives are bounded. We compute the derivative using Gauss theorem
\begin{align*}
\nabla \phi(|x|) = \int_B -\nabla_y \frac{1}{|x-y|^\lambda}\,dy = - \int_{\partial B} \frac{\nu_y}{|x-y|^\lambda}\,d\sigma(y) = - \int_{\Sph^{N-1}} \frac{\omega'}{|x-\omega'|}\,d\omega' \,.
\end{align*}
Thus, at $x=r\omega$ with $\omega\in\Sph^{N-1}$, $r\neq 1$, the negative of the radial derivative is equal to
\begin{align*}
-\partial_r\phi(r) = \int_{\Sph^{N-1}} \frac{\omega\cdot\omega'}{|r\omega-\omega'|^\lambda}\,d\omega' = |\Sph^{N-2}| \int_{-1}^1 \frac{t\, (1-t^2)^{(N-3)/2}}{(r^2 -2rt+1)^{\lambda/2}}\,dt \,.
\end{align*}
Thus, to prove the lemma we need to prove that the integral on the right side is uniformly bounded in $r\geq 0$. We bound $r^2-2rt +1 =(r-1)^2 +2r(1-t) \geq \max\{(r-1)^2,2r(1-t)\}$ and obtain
\begin{align*}
& \int_{-1}^1 \frac{t\,(1-t^2)^{(N-3)/2}}{(r^2 -2rt+1)^{\lambda/2}}\,dt
\leq \int_0^1 \frac{t (1-t^2)^{(N-3)/2}}{(r^2-2rt+1)^{\lambda/2}} \,dt \\
& \quad \leq \min\left\{ (r-1)^{-\lambda} \int_0^1 t (1-t^2)^{(N-3)/2} \,dt\,,\ (2r)^{-\lambda/2} \int_0^1 \frac{t(1+t)^{(N-3)/2}}{(1-t)^{(\lambda-N+3)/2}}\,dt \right\}.
\end{align*}
The assumption $\lambda<N-1$ implies that $(\lambda-N+3)/2<1$ and therefore the last integral is finite. This proves the boundedness of the derivative.

This boundedness proves the bound $|\phi(r)-\phi(1)|\leq C_{N,\lambda}|r-1|$ for all $0\leq r\leq 2$. On the other hand, for $r>2$, $0\leq \phi(1)-\phi(r)\leq \phi(1)\leq \phi(1)(r-1)$. This proves the claimed bound.
\end{proof}

\begin{remark}\label{phirem}
There are constants $c_{N,\lambda}>0$ such that
$$
-\phi'(r) \geq c_{N,\lambda} \times
\begin{cases}
|r-1|^{-\lambda+N-1} & \text{if}\ N-1<\lambda<N \,,\\
|\ln|1-r|| & \text{if}\ \lambda=N-1 \,,
\end{cases}
\qquad\text{for all}\ |r-1|\leq 1/2 \,.
$$
Indeed, we use the same expression for $\phi'$ as in the previous lemma and, using $r^2-2rt+1=(r-1)^2+2r(1-t)\leq 4r(1-t)$ for $t\leq 1-(1-r)^2/(2r)$ and $r^2+2rt+1\geq 1$, we obtain
\begin{align*}
& \int_{-1}^1 \frac{t\,(1-t^2)^{(N-3)/2}}{(r^2 -2rt+1)^{\lambda/2}}\,dt = \int_{0}^1 \left( \frac{t\,(1-t^2)^{(N-3)/2}}{(r^2 -2rt+1)^{\lambda/2}} - \frac{t\,(1-t^2)^{(N-3)/2}}{(r^2 +2rt+1)^{\lambda/2}}\right)dt \\
& \quad \geq \int_0^{1-(1-r)^2/(2r)} \left( \frac{t\,(1-t^2)^{(N-3)/2}}{(4r(1-t))^{\lambda/2}} - t\,(1-t^2)^{(N-3)/2}\right)dt \\
& \quad \geq (4r)^{-\lambda/2} \int_0^{1-(1-r)^2/(2r)} \frac{t\,(1+t)^{(N-3)/2}}{(1-t)^{(\lambda-N+3)/2}}\,dt  - \int_0^1 t\,(1-t^2)^{(N-3)/2}\,dt \,.
\end{align*}
The first integral on the right side is easily seen to diverge like $r^{-\lambda+N-1}$ if $\lambda>N-1$ and like $|\ln|r-1||$ if $\lambda=N-1$, while the second integral is bounded.
\end{remark}


\section{Proof of the main result}


\subsection{Step 0}

Let $E^*\subset\R^N$ denote the ball, centered at the origin, of volume $m$ and let
$$
\Phi(|x|) := \int_{E^*} \left( |x-y|^\alpha + |x-y|^{-\lambda}\right)dy \,.
$$
Note that the right side depends only on $|x|$, which justifies the notation. Moreover, let $R$ denote the radius of $E^*$. In this preliminary subsection we collect some bounds on $\Phi$ which show, in particular, that for $m$ large enough the characteristic function $\1_{E^*}$ satisfies the Euler--Lagrange conditions for the minimization problem $E_{\alpha,\lambda}(m)$.

\begin{lemma}\label{phialphalambda}
Let $0<\lambda<N-1$. Then there are $m_{N,\alpha,\lambda}<\infty$ and $c_{N,\alpha,\lambda}>0$ such that for all $m\geq m_{N,\alpha,\lambda}$ one has
\begin{equation}
\label{eq:el}
\Phi(r) \leq \Phi(R) \quad\text{if}\ r\leq R \,,
\qquad
\Phi(r)\geq \Phi(R) \quad\text{if}\ r\geq R \,,
\end{equation}
as well as
\begin{equation}
\label{eq:phialphalambdabound}
|\Phi(r)-\Phi(R)| \geq c_{N,\alpha,\lambda} R^{N+\alpha-1} \min\{|r-R|,R\} 
\qquad\text{for all}\ r\geq 0 \,.
\end{equation}
\end{lemma}

As shown in \cite[Lem.~4.2]{BCT}, the Euler--Lagrange conditions for $\rho$ to be a critical point of the optimization problem $E_{\alpha,\lambda}(m)$ is
\begin{align*}
(|x|^\alpha+|x|^{-\lambda})*\rho & \leq \mu \qquad \text{a.e. on}\ \{\rho=1\} \,,\\
(|x|^\alpha+|x|^{-\lambda})*\rho & = \mu \qquad \text{a.e. on}\ \{0<\rho<1\} \,,\\
(|x|^\alpha+|x|^{-\lambda})*\rho & \geq \mu \qquad \text{a.e. on}\ \{\rho=0\} \,,
\end{align*}
for some parameter $\mu>0$. For $\rho=\1_{E^*}$ these conditions simplify to $\Phi(r)\leq\mu$ a.~e.~on $\{0<r<R\}$ and $\Phi(r)\geq\mu$ a.~e.~on $\{r>R\}$, where $R$ is, as before, the radius of $E^*$. Since $\Phi$ is continuous, this can only hold with $\mu=\Phi(R)$, and Lemma \ref{phialphalambda} says that it does, indeed, for $m$ large enough.

The bound \eqref{eq:el} appears also in \cite[Lem.~5.5]{BCT}, but some details of the proof are omitted. In fact, as we show in Remark \ref{failure} below, this bound does not hold for $N-1\leq\lambda<N$. (Inequality \cite[(5.2)]{BCT} fails near $|x|=R$ if $\lambda\geq N-1$.)

\begin{proof}
We write $\phi_{-\lambda}$ for the potential defined in \eqref{eq:phi} with $\mathcal B$ being the unit ball and similarly $\phi_\alpha(|x|):= \int_{\mathcal B} |x-y|^\alpha\,dy$. Then, by scaling,
$$
\Phi(r) = R^{N+\alpha} \left( \phi_\alpha(r/R) + R^{-\alpha-\lambda} \phi_{-\lambda}(r/R) \right)
\qquad\text{for all}\ r\geq 0 \,.
$$
Clearly, $\phi_\alpha$ is monotone increasing and continuously differentiable. Therefore, there is a $c>0$, depending only on $N$ and $\alpha$, such that $\phi_\alpha'(s)\geq c$ for all $1/2\leq s\leq 3/2$. Also, by Lemma \ref{phi} there is a $C<\infty$, depending only on $N$ and $\lambda$, such that $0\leq -\phi_{-\lambda}'(s)\leq C$ for all $1/2\leq s\leq 3/2$. Thus, for all $1\leq s\leq 3/2$,
\begin{align*}
(\phi_\alpha(s)-\phi_\alpha(1)) + R^{-\alpha-\lambda} (\phi_{-\lambda}(s)-\phi_{-\lambda}(1))
& = \int_1^s \left( \phi_\alpha'(t) + R^{-\alpha-\lambda} \phi_{-\lambda}'(t)\right)dt \\
& \geq (c- R^{-\alpha-\lambda} C) (s-1) \,.
\end{align*}
Thus, if we choose $m$ so large that $R^{-\alpha-\lambda} C \leq c/2$, we can bound the right side from below by $(c/2)(s-1)$. Similarly, one shows that for all $1/2\leq s\leq 1$,
$$
(\phi_\alpha(1)-\phi_\alpha(s)) + R^{-\alpha-\lambda} (\phi_{-\lambda}(1)-\phi_{-\lambda}(1))
\geq (c/2) (1-s).
$$
Now if $s\geq 3/2$, we use $\phi_\alpha(s)\geq \phi_\alpha(3/2)$ and $\phi_{-\lambda}(s)\geq 0$ to bound
\begin{align*}
(\phi_\alpha(s)-\phi_\alpha(1)) + R^{-\alpha-\lambda} (\phi_{-\lambda}(s)-\phi_{-\lambda}(1))
& \geq \phi_\alpha(3/2)-\phi_\alpha(1) - R^{-\alpha-\lambda} \phi_{-\lambda}(1) \\
& \geq (c/2)(3/2-1) - R^{-\alpha-\lambda} \phi_{-\lambda}(1) \,.
\end{align*}
Increasing $m$ if necessary we can assume that $R^{-\alpha-\lambda}\phi_{-\lambda}(1)\leq c/8$ and then the right side is bounded from below by $c/8$. Similarly, for all $s\leq 1/2$,
$$
(\phi_\alpha(1)-\phi_\alpha(s)) + R^{-\alpha-\lambda} (\phi_{-\lambda}(1)-\phi_{-\lambda}(s)) \geq c/8 \,,
$$
provided $R^{-\alpha-\lambda}(\phi_{-\lambda}(1)-\phi_{-\lambda}(0))\leq c/8$.

After rescaling, the above inequalities become
$$
\Phi(r) - \Phi(R) \geq
\begin{cases}
\tfrac c2 R^{N+\alpha-1} (r-R) & \text{if}\ R\leq r\leq \tfrac32 R \,,\\
\tfrac c8 R^{N+\alpha} & \text{if} R > \tfrac32 R \,,
\end{cases}
$$
and
$$
\Phi(R) - \Phi(r) \geq
\begin{cases}
\tfrac c2 R^{N+\alpha-1} (R-r) & \text{if}\ \tfrac12 R\leq r\leq R \,,\\
\tfrac c8 R^{N+\alpha} & \text{if} R < \tfrac12 R \,,
\end{cases}
$$
This proves both statements in the lemma.
\end{proof}

\begin{remark}\label{failure}
We claim that if $N-1\leq\lambda<N$, then for any $m>0$ there are $r_1<R<r_2$ such that
$$
\Phi(r)> \Phi(R) >\Phi(r')
\qquad\text{for all}\ r_1<r<R<r'<r_2 \,.
$$
Consequently for $N-1\leq\lambda<N$, although the Euler--Lagrange conditions always have a solution, they are never satisfied by the characteristic function of a ball.

To prove this, we note that, since $\phi_\alpha$ is continuously differentiable, there is a $C<\infty$, depending only on $N$ and $\alpha$, such that $0\leq\phi_\alpha'(s)\leq C$ for all $1/2\leq s\leq 3/2$. On the other hand, by Remark \ref{phirem} there is a $c>0$, depending only on $N$ and $\lambda$, such that $-\phi_{-\lambda}'(s)\geq c|s-1|^{-\lambda+N-1}$. (We assume here $\lambda>N-1$, the case $\lambda=N-1$ is handled similarly.) Thus, for all $1\leq s\leq 3/2$,
\begin{align*}
(\phi_\alpha(s)-\phi_\alpha(1)) + R^{-\alpha-\lambda} (\phi_{-\lambda}(s)-\phi_{-\lambda}(1))
& = \int_1^s \left( \phi_\alpha'(t) + R^{-\alpha-\lambda} \phi_{-\lambda}'(t)\right)dt \\
& \leq \int_1^s \left( C - c R^{-\alpha-\lambda} (t-1)^{-\lambda+N-1} \right)dt \\
& = C(s-1) - \frac{c}{N-\lambda} R^{-\alpha-\lambda} (s-1)^{N-\lambda}
 \,.
\end{align*}
Clearly, for any $R$, the right side is negative in a right neighborhood of $s=1$. The same argument shows that for all $1/2\leq s\leq 1$,
$$
(\phi_\alpha(1)-\phi_\alpha(s)) + R^{-\alpha-\lambda} (\phi_{-\lambda}(1)-\phi_{-\lambda}(s)) \leq C(1-s) - \frac{c}{N-\lambda} R^{-\alpha-\lambda} (1-s)^{N-\lambda}
$$
and the right side is negative in a left neighborhood of $s=1$. This proves the claim.
\end{remark}


\subsection{Step 1}

We now begin with the main part of the proof of Theorem \ref{main}. In this step we show that for large $m$, minimizers are close to characteristic functions of balls. This closeness is expressed in the sense of the quantity $A[\cdot]$, that is, in $L^1$-norm.

\begin{proposition}\label{step1}
There is a constant $C_{N,\alpha,\lambda}<\infty$ such that, if $\rho$ is a minimizer corresponding to $E_{\alpha,\lambda}(m)$ with $m>0$, then
$$
A[\rho] \leq C_{N,\alpha,\lambda}\, m^{-(\alpha+\lambda)/N} \,.
$$
\end{proposition}

We emphasize that in the following we only need the much weaker fact that $A[\rho]\to 0$ as $m\to\infty$, which could also be proved, for instance, using compactness. The following proof has the advantange of giving (in principle) a computable constant and also of introducing a technique that we will use again later in Step 3.

\begin{proof}
Let $B$ be a ball of measure $m$. According to Theorem \ref{christcor2},
\begin{equation}
\label{eq:step1proof1}
c_{N,\alpha} \, m^{2+\alpha/N} A[\rho]^2 \leq \mathcal I_\alpha[\rho] - \mathcal I_\alpha[\1_{B}] \,.
\end{equation}
Further, since $\rho$ is a minimizer,
\begin{equation}
\label{eq:step1proof2}
\mathcal I_\alpha[\rho] - \mathcal I_\alpha[\1_{B}] \leq \mathcal I_{-\lambda}[\1_{B}] - \mathcal I_{-\lambda}[\rho] \,.
\end{equation}
We estimate, using the bathtub principle \cite[Thm.~1.14]{LiLo},
\begin{align*}
\mathcal I_{-\lambda}[\1_{B}]- \mathcal I_{-\lambda}[\rho]
& = \mathcal I_{-\lambda}[\1_{B}-\rho,\1_{B}+\rho] \\
& \leq \frac12 \|\1_{B}-\rho \|_1 \left( \||x|^{-\lambda} * \1_B \|_\infty + \| |x|^{-\lambda} * \rho \|_\infty \right) \\
& \leq \| \1_B - \rho \|_1 \||x|^{-\lambda} * \1_B \|_\infty \\
& \leq C_{N,\lambda} m^{1-\lambda/N} \| \1_B - \rho \|_1 \,.
\end{align*}
Combining this bound with \eqref{eq:step1proof1} and \eqref{eq:step1proof2} we obtain
$$
c_{N,\alpha} \, m^{2+\alpha/N} A[\rho]^2 \leq C_{N,\lambda} m^{1-\lambda/N} \| \1_B-\rho \|_1 \,,
$$
and therefore, taking the infimum over all $B$'s,
$$
c_{N,\alpha} \, m^{2+\alpha/N} A[\rho]^2 \leq 2 C_{N,\lambda} m^{2-\lambda/N} A[\rho] \,.
$$ 
This proves the claimed bound.
\end{proof}


\subsection{Step 2}
This is the key step of the proof! We show that for large $m$, minimizers $\rho$ are close to characteristic functions of balls not only in $L^1$ sense (as shown in Step 1), but also in the Hausdorff sense. More precisely, they differ from the characteristic function of a ball only in a shell around this ball of relative width at most of the order of $A[\rho]$. As explained after Proposition \ref{christcor1rev}, it is crucial to get precisely this bound of the shell width. 

\begin{proposition}\label{step2}
There are constants $C_{N,\alpha,\lambda}<\infty$ and $m_{N,\alpha,\lambda}<\infty$ such that any minimizer $\rho$ corresponding to $E_{\alpha,\lambda}(m)$ with $m\geq m_{N,\alpha,\lambda}$ there is a ball $B\subset\R^N$ of measure $m$ with
$$
\1_{(1-C_{N,\alpha,\lambda}A[\rho])B} \leq \rho \leq \1_{(1+C_{N,\alpha,\lambda}A[\rho])B} \,.
$$
\end{proposition}

For the proof of this proposition we need several preliminary lemmas. The first one is an extension of a construction in \cite[Sec.~5]{Ch}.

\begin{lemma}\label{competitor}
Let $\rho\in L^1(\R^N)$ with $0\leq\rho\leq 1$. Let $B$ be a ball of measure $\int_{\R^N}\rho\,dx$ and let $0\leq\theta\leq 1$. Then there is a $\tilde\rho\in L^1(\R^N)$ with
\begin{equation}
\label{eq:rhotilde1}
\int_{\R^N}\rho'\,dx = \int_{\R^N} \rho\,dx \,,
\end{equation}
\begin{equation}
\label{eq:rhotilde2}
\1_{(1-\theta)B} \leq \tilde \rho \leq \1_{(1+\theta)B} \,,
\end{equation}
\begin{equation}
\label{eq:rhotilde3}
\tilde \rho \geq \rho
\quad\text{in}\ B
\qquad\text{and}\qquad
\tilde \rho \leq \rho
\quad\text{in}\ \R^N\setminus B \,,
\end{equation}
\begin{equation}
\label{eq:rhotilde5}
\int_{\R^N} |\tilde\rho-\1_B| \,dx \leq \int_{\R^N} |\rho-\1_B|\,dx
\end{equation}
and
\begin{equation}
\label{eq:rhotilde4}
\int_{(1-\theta)B \cup(\R^N\setminus (1+\theta)B)} |\tilde\rho-\rho| \,dx \geq \frac12 \int_{\R^N} |\tilde\rho-\rho|\,dx \,.
\end{equation}
\end{lemma}

\begin{proof}
By translation and scale invariance, we may assume that $B$ is the ball of radius~$1$ centered at the origin. Let
$$
m_i:= \int_{\{|x|<1-\theta\}} (1-\rho)\,dx
\qquad\text{and}\qquad
m_o:= \int_{\{|x|>1+\theta\}} \rho\,dx \,.
$$

If $m_i\geq m_o$, we choose $r_o$ such that
$$
\int_{\{|x|>r_o\}} \rho\,dx = m_i
$$
and note that $r_o\leq 1+\theta$. On the other hand, $r_0\geq 1$ since, using the fact that $\rho-\1_B$ has integral zero,
$$
\int_{\{|x|>1\}} \rho\,dx = \int_{\R^N} (\rho-\1_B)_+\,dx = \int_{\R^N} (\rho-\1_B)_-\,dx = \int_B (1-\rho)\,dx \geq m_i \,.
$$
We set
$$
\tilde\rho := \rho\1_{\{|x|\leq r_o\}} + (1-\rho)\1_{\{|x|\leq 1-\theta\}} \,.
$$

If $m_i<m_o$, we choose $r_i$ such that
$$
\int_{\{|x|<r_i\}} (1-\rho)\,dx = m_o
$$
and note that $r_i\geq 1-\theta$. On the other hand, $r_i\leq 1$ by the same computation that showed $r_o\geq 1$ in the first case. We set
$$
\tilde\rho := \rho\1_{\{|x|\leq 1+\theta\}} + (1-\rho)\1_{\{|x|\leq r_i\}} \,.
$$

In both cases, the properties \eqref{eq:rhotilde1}, \eqref{eq:rhotilde2} and \eqref{eq:rhotilde3} follow immediately from the construction. Moreover, property \eqref{eq:rhotilde5} follows immediately from \eqref{eq:rhotilde3}. In order to prove \eqref{eq:rhotilde4} we set
$$
A := \{ 1-\theta \leq |x| <1+\theta\} \,,
$$
so \eqref{eq:rhotilde4} is equivalent to
$$
\int_{\R^N\setminus A} |\tilde\rho-\rho|\,dx \geq \int_A |\tilde\rho-\rho|\,dx \,.
$$
To unify the treatment of the two cases we set $\rho_i=1-\theta$ if $m_i\geq m_o$ and $\rho_o=1+\theta$ if $m_i<m_o$, so that in both cases
$$
\tilde\rho = \1_{\{|x|\leq r_i \}} + \rho \1_{\{r_i<|x|\leq r_o\}} \,.
$$
Thus,
$$
\int_A |\tilde\rho-\rho|\,dx = \int_{\{1-\theta\leq|x|< r_i\}}(1-\rho)\,dx + \int_{\{r_o\leq |x|< 1+\theta\}} \rho\,dx \,.
$$
We claim that
$$
\int_A |\tilde\rho-\rho|\,dx \leq \max\{m_i,m_o\} \,.
$$
Indeed, if $m_i\geq m_o$, then the set $\{1-\theta\leq|x|< r_i\}$ is empty and
$$
\int_{\{r_o\leq |x|< 1+\theta\}} \rho\,dx = m_i - \int_{\{|x|\geq 1+\theta\}} \rho\,dx \leq m_i = \max\{m_i,m_o\} \,,
$$
and similarly if $m_i<m_o$. On the other hand,
$$
\int_{\R^N\setminus A} |\tilde\rho-\rho|\,dx = \int_{\{|x|<1-\theta\}}(1-\rho)\,dx + \int_{\{|x|\geq 1+\theta\}} \rho\,dx = m_i + m_o \geq \max\{m_i,m_o\} \,.
$$
This proves \eqref{eq:rhotilde4} and completes the proof.
\end{proof}

The following lemma gives a bound on the $\lambda$-potential of a function $0\leq\rho\leq 1$. To motivate the bound, we note that
\begin{equation}
\label{eq:coulombstandard}
\sup_{x\in\R^N} \int_{\R^N} \frac{\rho(y)}{|x-y|^\lambda} \,dy \leq C_{N,\lambda} \left( \int_{\R^N} \rho\,dx \right)^{1-\lambda/N} \,.
\end{equation}
This bound, which we used in the proof of Proposition \ref{step1}, follows from the bathtub principle \cite[Thm.~1.14]{LiLo}. Indeed, the latter implies that to make the integral on the left side with a given $x\in\R^N$ as large as possible one takes $\rho$ to be the characteristic function of a ball centered at $x$ and of measure $\|\rho\|_1$. In particular, the radius of this ball is $c_N \|\rho\|_1^{1/N}$. We now show that the bound \eqref{eq:coulombstandard} can be improved, provided the support of $\rho$ is contained in an annular shell of a width $\theta R$ which is smaller than the radius $c_N \|\rho\|_1^{1/N}$.

\begin{lemma}\label{coulombimproved}
Let $0<\lambda<N-1$. Then there is a constant $C_{N,\lambda}<\infty$ such that for any $\rho\in L^1(\R^N)$ with $0\leq\rho\leq 1$ and any $R>0$, $0\leq\theta\leq 1$ with
$$
\supp\rho\subset\{(1-\theta)R\leq |x|\leq (1+\theta) R\}
$$
one has
$$
\sup_{x\in\R^N} \int_{\R^N} \frac{\rho(y)}{|x-y|^\lambda} \,dy \leq C_{N,\lambda} (\theta R)^{\lambda/(N-1)} \left( \int_{\R^N} \rho\,dx \right)^{1-\lambda/(N-1)}.
$$ 
\end{lemma}

\begin{proof}
Because of the bound \eqref{eq:coulombstandard} it suffices to prove the lemma under the additional assumption
\begin{equation}
\label{eq:coulombimprovedproof}
\theta R \leq \epsilon_N \left( \int_{\R^N} \rho\,dx \right)^{1/N}
\end{equation}
for some fixed constant $\epsilon_N>0$, depending only on $N$, to be specified later. We abbreviate $A:=\{ (1-\theta)R\leq |x|\leq (1+\theta)R\}$. Fix $x\in\R^N$ and define $r>0$ by
$$
| A\cap B_r(x) | = \int_{\R^N} \rho\,dx \,.
$$
Then, by the bathtub principle \cite[Thm.~1.14]{LiLo},
$$
\int_{\R^N} \frac{\rho(y)}{|x-y|^\lambda} \,dy \leq \int_{A\cap B_r(x)} \frac{dy}{|x-y|^\lambda} \,.
$$
We are left with bounding the integral on the right side. Note that, by \eqref{eq:coulombimprovedproof},
$$
\theta R \leq \epsilon_N |A\cap B_r(x)|^{1/N} \leq \epsilon_N \min\{|A|^{1/N},|B_r(x)|^{1/N}\} \,.
$$
Since $|B_r(x)|^{1/N}=C_N r$ and $|A|^{1/N} \leq C_N' R\,\theta^{1/N}$, we may assume that
$$
\theta R \leq \epsilon_N' r
\qquad\text{and}\qquad
\theta \leq \epsilon_N''
$$
for constants $\epsilon_N',\epsilon_N''>0$ which can be chosen arbitrarily small, depending only on $N$.

A consequence of these bounds is that the curvature of the annular region $A$ is negligible and that we can, within controlled factors, replace $A\cap B_r(x)$ by a set of the form $x+([-L,L]^{N-1}\times[-\ell,\ell])$, where $L\sim r$ and $\ell \sim \theta R$. (The notation $\sim$ here means that the quotient of the two quantities is bounded from above and from below by a constant depending only on $N$.) We thus have
\begin{align*}
\int_{A\cap B_r(x)} \frac{dy}{|x-y|^\lambda}
& \lesssim \iint_{[-L,L]^{N-1}\times[-\ell,\ell]} \frac{dy'\,dy_N}{((y')^2+y_N^2)^{\lambda/2}} \\
& = L^{N-\lambda-1} \ell \iint_{[-1,1]^{N-1}\times[-1,1]} \frac{dz'\,dz_N}{((z')^2+(\ell/L)^2 z_N^2)^{\lambda/2}} \\
& \leq 2 L^{N-\lambda-1} \ell \int_{[-1,1]^{N-1}} \frac{dz'}{|z'|^{\lambda}} \,.
\end{align*}
The latter integral is finite since $\lambda<N-1$. Finally, since
$$
L^{N-\lambda-1}\ell = (L^{N-1}\ell)^{1-\lambda/(N-1)} \ell^{\lambda/(N-1)} \sim |A\cap B_r(x)|^{1-\lambda/(N-1)} (\theta R)^{\lambda/(N-1)} \,,
$$
we obtain the claimed bound.
\end{proof}

The next lemma gives a bound on the diameter of a minimizer.

\begin{lemma}\label{diameter}
For any $\alpha>0$ and $0<\lambda<N$ there is a constant $C_{N,\alpha,\lambda}<\infty$ such that any minimizer $\rho$ corresponding to $E_{\alpha,\lambda}(m)$ with $m\geq 1$ satisfies
$$
\diam\supp\rho \leq C_{N,\alpha,\lambda} m^{1/N} \,.
$$
\end{lemma}

This bound was proved in \cite[Thm.~4.1]{FL} in the case $N=3$, $\lambda=1$. The same proof extends to the more general situation considered here. We omit the details.

Finally, we can prove the main result of this subsection.

\begin{proof}[Proof of Proposition \ref{step2}]
Let $\rho$ be a minimizer corresponding to $E_{\alpha,\lambda}(m)$. The overall strategy is to construct a competitor with the desired support properties and then deduce by minimality of $\rho$ that $\rho$ has to coincide with this competitor.

Let $B$ be a ball of measure $\int_{\R^N}\rho\,dx$ such that
\begin{equation}
\label{eq:aattained}
\|\rho - \1_B \|_1 = 2 \,\|\rho\|_1\, A[\rho] \,.
\end{equation}
(Such a ball exists, since $a\mapsto\| \rho - \1_{E^*+a}\|_1$ is continuous, tends to $2\|\rho\|_1$ at infinity and assumes somewhere a value strictly less than $2\|\rho\|_1$.) By translation invariance we may assume that $B$ is centered at the origin and we write $B=E^*$.

We construct successively a sequence of functions $\rho_n$, $n\geq -1$, as follows. We set $\rho_{-1}:=\rho$. If $\rho_{n-1}$ is already constructed for some $n\geq 0$, then $\rho_n$ is defined to be the $\tilde\rho$ from Lemma \ref{competitor} with $\rho_{n-1}$ in place of $\rho$, with $\theta=2^{-n}$ and with the given ball $E^*$.

We write, for $n\geq 0$,
\begin{align}
\label{eq:expansionenergy}
\mathcal E_{\alpha,\lambda}[\rho_n] - \mathcal E_{\alpha,\lambda}[\rho_{n-1}] 
& = \mathcal E_{\alpha,\lambda}[\rho_n-\rho_{n-1},\rho_n+\rho_{n-1}] \notag \\
& = 2\, \mathcal E_{\alpha,\lambda}[\rho_n-\rho_{n-1},\1_{E^*}] + \mathcal E_{\alpha,\lambda}[\rho_n-\rho_{n-1},\rho_n+\rho_{n-1}-2\cdot\1_{E^*}] \,.
\end{align}

We begin with the first term on the right side of \eqref{eq:expansionenergy}, which is the main term. Recall that $\Phi$ was defined in the proof of Lemma \ref{phialphalambda} and that $R$ denotes the radius of $E^*$. Using the fact that, by \eqref{eq:rhotilde1}, $\rho_n- \rho_{n-1}$ has integral zero and properties \eqref{eq:el} and \eqref{eq:rhotilde3} we obtain
$$
2\, \mathcal E_{\alpha,\lambda}[\rho_n-\rho_{n-1},\1_{E^*}] = \int_{\R^N}\! (\rho_n - \rho_{n-1})(\Phi-\Phi(R))\,dx = - \int_{\R^N} \!\left|\rho_n - \rho_{n-1} \right| \left|\Phi-\Phi(R)\right| dx.
$$
Using the bound \eqref{eq:phialphalambdabound} from Lemma \ref{phialphalambda} we conclude that, using \eqref{eq:rhotilde4},
\begin{align*}
\int_{\R^N} \left|\rho_n - \rho_{n-1} \right| \left|\phi-\phi(R)\right| dx & \geq c_{N,\alpha,\lambda} R^{N+\alpha} 2^{-n} \int_{\{ \left||x|-R\right|\geq 2^{-n} R\}} |\rho_n-\rho_{n-1}|\,dx \\
& \geq \frac12 c_{N,\alpha,\lambda} R^{N+\alpha} 2^{-n} \|\rho_n-\rho_{n-1}\|_1 \,.
\end{align*}

We now turn our attention to the second term on the right side of \eqref{eq:expansionenergy}, which is a remainder term. We decompose $\mathcal E_{\alpha,\lambda}=\mathcal I_\alpha + \mathcal I_{-\lambda}$. For the term involving $\alpha$ we use the diameter bound from Lemma \ref{diameter}. Note that by construction of $\rho_n$, the same bound holds also for $\supp\rho_n$, with a constant independent of $n$. We obtain, using \eqref{eq:rhotilde5},
\begin{align*}
\mathcal I_{\alpha}[\rho_n-\rho_{n-1},\rho_n+\rho_{n-1}-2\cdot\1_{E^*}]
& \leq C_{N,\alpha,\lambda}^\alpha m^{\alpha/N} \|\rho_n-\rho_{n-1}\|_1 \|\rho_n+\rho_{n-1} - 2\cdot\1_{E^*}\|_1 \\
& \leq C_{N,\alpha,\lambda}' R^{\alpha} \|\rho_n-\rho_{n-1}\|_1 \|\rho_{n-1}-\1_{E^*}\|_1 \,.
\end{align*}

Finally, for the term involving $\lambda$ we write
\begin{align*}
& \mathcal I_{-\lambda}[\rho_n-\rho_{n-1},\rho_n+\rho_{n-1}-2\cdot\1_{B}] \\
& = - \mathcal I_{-\lambda}[\rho_n-\rho_{n-1},\rho_n-\rho_{n-1}] + 2\,\mathcal I_{-\lambda}[\rho_n-\rho_{n-1},\rho_n-\1_{E^*}] \,.
\end{align*}
Since convolution with $|x|^{-\lambda}$ is positive semi-definite, we have
$$
- \mathcal I_{-\lambda}[\rho_n-\rho_{n-1},\rho_n-\rho_{n-1}]\leq 0 \,.
$$
Moreover,
$$
2\,\mathcal I_{-\lambda}[\rho_n-\rho_{n-1},\rho_n-\1_{E^*}] \leq \|\rho_n-\rho_{n-1}\|_1 \left\| |x|^{-\lambda} * (\rho_n-\1_{E^*}) \right\|_\infty \,.
$$
By \eqref{eq:rhotilde2}, the support of $\rho_n-\1_{E^*}$ is contained in $\{ (1-2^{-n})R\leq |x|\leq (1+2^{-n})R\}$. Thus, by Lemma \ref{coulombimproved}, we obtain
\begin{align*}
2\,\mathcal I_{-\lambda}[\rho_n-\rho_{n-1},\rho_n-\1_{E^*}] 
& \leq C_{N,\lambda} (2^{-n} R)^{\lambda/(N-1)} \|\rho_n-\rho_{n-1}\|_1 \|\rho_n-\1_{E^*}\|_1^{1-\lambda/(N-1)} \\
& \leq C_{N,\lambda} (2^{-n} R)^{\lambda/(N-1)} \|\rho_n-\rho_{n-1}\|_1 \|\rho_{n-1}-\1_{E^*}\|_1^{1-\lambda/(N-1)}.
\end{align*}
Here we used again \eqref{eq:rhotilde5}.

To summarize, we have shown that
\begin{align*}
& \mathcal E_{\alpha,\lambda}[\rho_n] - \mathcal E_{\alpha,\lambda}[\rho_{n-1}] \\
& \leq - 2^{-n} R^{N+\alpha} \|\rho_n-\rho_{n-1}\|_1 
\left( \frac12 c_{N,\alpha,\lambda} - C_{N,\alpha,\lambda}' \epsilon_n - C_{N,\lambda} R^{-\alpha-\lambda} \epsilon_n^{1-\lambda/(N-1)} \right)
\end{align*}
with
$$
\epsilon_n := 2^n R^{-N} \|\rho_{n-1}-\1_{E^*}\|_1 \,.
$$
In particular,
\begin{align}
\label{eq:step2proof}
& \mathcal E_{\alpha,\lambda}[\rho_{n_0}] - \mathcal E_{\alpha,\lambda}[\rho] = \sum_{n=0}^{n_0} \left( \mathcal E_{\alpha,\lambda}[\rho_{n}] - \mathcal E_{\alpha,\lambda}[\rho_{n-1}] \right) \notag \\
& \leq - \sum_{n=0}^{n_0} 2^{-n} R^{N+\alpha} \|\rho_n-\rho_{n-1}\|_1 \left( \frac12 c_{N,\alpha,\lambda} - C_{N,\alpha,\lambda}' \epsilon_n - C_{N,\lambda} R^{-\alpha-\lambda} \epsilon_n^{1-\lambda/(N-1)} \right).
\end{align}

According to Proposition \ref{step1} and the choice \eqref{eq:aattained} of the ball $B=E^*$ there is an $m_{N,\alpha,\lambda}<\infty$ such that for all $m\geq m_{N,\alpha,\lambda}$,
$$
\epsilon_0 = 2\, |\{ |x|<1\}|\, A[\rho] < \frac14 \frac{c_{N,\alpha,\lambda}}{C_{N,\alpha,\lambda}'} \,.
$$
To proceed, we assume first that there is a non-negative integer $n_0$ such that
$$
\epsilon_n < \frac14 \frac{c_{N,\alpha,\lambda}}{C_{N,\alpha,\lambda}'}
\qquad\text{for all}\ n=0,\ldots,n_0
\qquad\text{and}\qquad
\epsilon_{n_0+1} \geq \frac14 \frac{c_{N,\alpha,\lambda}}{C_{N,\alpha,\lambda}'} \,.
$$
Increasing $m_{N,\alpha,\lambda}$ if necessary, we may assume that for all $m\geq m_{N,\alpha,\lambda}$,
$$
C_{N,\lambda} R^{-\alpha-\lambda} \left( \frac14 \frac{c_{N,\alpha,\lambda}}{C_{N,\alpha,\lambda}'} \right)^{1-\lambda/(N-1)} \leq \frac14 c_{N,\alpha,\lambda} \,.
$$
This implies that if $m\geq m_{N,\alpha,\lambda}$ and $n\leq n_0$, then
$$
\frac12 c_{N,\alpha,\lambda} - C_{N,\alpha,\lambda}' \epsilon_n - C_{N,\lambda} R^{-\alpha-\lambda} \epsilon_n^{1-\lambda/(N-1)} > 0 \,.
$$
With this information we return to \eqref{eq:step2proof}. Since $\rho$ is a minimizer, we have
$$
\mathcal E_{\alpha,\lambda}[\rho_{n_0}] - \mathcal E_{\alpha,\lambda}[\rho] \geq 0 \,.
$$
Therefore all the terms in the sum on the right side of \eqref{eq:step2proof} have to vanish, which means that $\rho_n=\rho_{n-1}$ for all $n=0,\ldots,n_0$. Thus, $\rho_{n_0}=\rho$. By \eqref{eq:rhotilde2} we have
$$
\1_{(1-2^{-n_0})E^*} \leq \rho \leq \1_{(1+2^{-n_0})E^*} \,.
$$
The lower bound on $\epsilon_{n_0+1}$, together with \eqref{eq:aattained}, implies
$$
2^{-n_0} \leq 8 \frac{C_{N,\alpha,\lambda}'}{c_{N,\alpha,\lambda}} R^{-N} \|\rho_{n_0} - \1_{E^*}\|_1
= 16 \frac{C_{N,\alpha,\lambda}'}{c_{N,\alpha,\lambda}} |\{|x|<1\}|\, A[\rho] \,.
$$
This completes the proof of the proposition in the case where $n_0$ exists.

Otherwise, the inequality $\epsilon_n < c_{N,\alpha,\lambda}/ (4\, C_{N,\alpha,\lambda}')$ holds for all $n$, and then by the same argument as above we conclude that $\rho=\rho_n$ for all $n$. This means that $\rho=\1_{E^*}$, so the proposition holds in this case as well.
\end{proof}


\subsection{Step 3}
We now complete the proof of our main result.

\begin{proof}[Proof of Theorem \ref{main}]
We choose $B$ to be the ball from Proposition \ref{step2}. Then this proposition guarantees that the assumption of Proposition \ref{christcor1rev} is satisfied with $\theta = C_{N,\alpha,\lambda} A[\rho]$. We have $\theta\leq 1$ by Proposition \ref{step1} for $m$ large enough, depending only on $N$, $\alpha$ and $\lambda$. Proposition \ref{christcor1rev} implies that
$$
\mathcal I_{-\lambda}[\1_B] - \mathcal I_{-\lambda}[\rho] \leq C_{N,\lambda} C_{N,\alpha,\lambda}^2 m^{2-\lambda/N} A[\rho]^2 \,.
$$
If we combine this inequality with \eqref{eq:step1proof1} and \eqref{eq:step1proof2} from the proof of Proposition~\ref{step1} we obtain
$$
c_{N,\alpha} m^{2+\alpha/N} A[\rho]^2 \leq C_{N,\lambda} C_{N,\alpha,\lambda}^2 m^{2-\lambda/N} A[\rho]^2 \,.
$$
This implies
$$
A[\rho]=0
\qquad\text{if}\quad m^{(\alpha+\lambda)/N} > c_{N,\alpha}^{-1} C_{N,\lambda} C_{N,\alpha,\lambda}^2 \,,
$$
which completes the proof.
\end{proof}



\bibliographystyle{amsalpha}

\end{document}